\documentclass[a4paper,11pt]{article}
\usepackage[english]{babel}
\usepackage[utf8]{inputenc}
\usepackage{paralist,url,verbatim}
\usepackage{amscd}
\usepackage{mathptmx}
\usepackage{amsmath,amsthm,amssymb,amsfonts}
\usepackage[bookmarksopen=true]{hyperref}

\theoremstyle{plain}
\newtheorem{theorem}{\bf Theorem}[section]

\newtheorem{corollary}[theorem]{Corollary}
\newtheorem{lemma}[theorem]{Lemma}
\newtheorem{proposition}[theorem]{Proposition}

\newtheorem{question}[theorem]{Question}

\theoremstyle{definition}
\newtheorem{remark}[theorem]{Remark}
\newtheorem{definition}[theorem]{Definition}

\newtheorem{example}[theorem]{Example}

\newcommand{\depth}{\operatorname{depth} }

\newcommand{\grade}{\operatorname{grade} }
\newcommand{\pdim}{\operatorname{projdim} }

\newcommand{\chara}{\operatorname{char} }
\newcommand{\QQ}{\mathbb{Q}}

\newcommand{\N}{\mathbb{N}}
\newcommand{\ZZ}{\mathbb{Z}}

\newcommand{\NN}{\mathbb{N}}
\newcommand{\nn}{\mathfrak{n}}
\newcommand{\mm}{\mathfrak{m}}

\newcommand{\height}{\operatorname{height} }

\begin{document}

\title{When does depth stabilize early on?}
\author{
Le Dinh Nam \thanks{Supported by VIASM-2014}\\
\small School of Applied Mathematics and Informatics,\\
\small Hanoi University of Science and Technology\\
\small \url{nam.ledinh@hust.edu.vn}
\and 
Matteo Varbaro \thanks{Supported by PRIN  2010S47ARA\_003 ``Geometria delle Variet\`a Algebriche".} \\
\small Dipartimento di Matematica, \\ Universit\`a di Genova \\
\small \url{varbaro@dima.unige.it}
\and
{\it To Ngo Viet Trung on his 60th birthday}
}
\date{}
\maketitle

\begin{abstract}
\noindent In this paper we study graded ideals $I$ in a polynomial ring $S$ such that the numerical function $k\mapsto \depth(S/I^k)$ is constant. We show that, if (i) the Rees algebra of $I$ is Cohen-Macaulay, (ii) the cohomological dimension of $I$ is not larger than the projective dimension of $S/I$ and (iii) the $K$-algebra generated by some homogeneous generators of $I$ is a direct summand of $S$, then $\depth(S/I^k)$ is constant. All the ideals with constant depth-function discovered by Herzog and Vladoiu in \cite{HV} satisfy the criterion given
above. In the contest of square-free monomial ideals, there is a chance that a converse of the previous fact holds true.
\end{abstract}

\section{Introduction}

Let $S$ be a polynomial ring in $n$ variables over a field $K$, and $I\subseteq S$ a homogeneous ideal. In this paper we study the {\it depth-function} of $I$:
\[k\mapsto \depth(S/I^k).\]
By a classical result of Brodmann \cite{Bro}, there exists $k_0\in\NN$ such that 
\[\depth(S/I^k)=\depth(S/I^{k_0}) \ \ \ \forall \ k\geq k_0.\]
In other words, depth-functions are definitely constant. Though, their initial behavior is hard to understand (for example see \cite{HH}). The purpose of this work is to inquire on ideals with constant depth-function, i.e. such that
\[\depth(S/I^k)=\depth(S/I) \ \ \ \forall \ k\geq 1.\]
Note that, if $\dim(S/I)=0$, then the depth-function of $I$ is obviously constant. On the other hand, if we assume that $I$ is {\it radical} the situation is much more rigid: if $S/I$ is Cohen-Macaulay and $I$ is radical, then $I$ has constant depth-function if and only if $I$ is a complete intersection, by a result of Cowsik and Nori in \cite{CN}. The main result of the present paper is Theorem \ref{thm:main}, where a class of ideals with constant depth-function is identified. Precisely, we show that, if $I$ is a homogeneous ideal of $S$ generated by $f_1,\ldots ,f_r$ such that:
\begin{itemize}
\item[(i)] the Rees algebra of $I$ is Cohen-Macaulay;
\item[(ii)] $H_I^i(S)=0$ for any $i>\pdim(S/I)$;
\item[(iii)] $A=K[f_1,\ldots ,f_r]$ is a direct summand of $S$ (as an $A$-module);
\end{itemize}
then the depth-function of $I$ is constant.
The above hypotheses are interesting by themselves. We note that (ii) is satisfied by a broad class of ideals in Proposition \ref{prop:notsorare} and that (iii), which at first sight might seem stronger than (i), does not imply the latter in general (Example \ref{ex:no}).

In the last section we restrict ourselves to consider monomial ideals. For this kind of ideals condition (ii) is automatically satisfied, whereas (i) and (iii) are still independent, as shown by Example \ref{ex:no}. We discuss when $A=K[u_1,\ldots ,u_r]$ is a direct  summand of $S$ when $u_1,\ldots ,u_r$ are monomial generators of $I$, and especially we report a characterization of when $A$ is an algebra retract of $S$, that we learnt on MathOverflow (Lemma \ref{lem:retract}). The problem of characterizing monomial ideals with constant depth-function was already addressed by Herzog and Vladoiu in \cite{HV}, where they provided large classes of square-free monomial ideals with constant depth-function. All such ideals satisfy the assumptions of Theorem \ref{thm:main}. They also gave examples of square-free monomial ideals with constant depth-function lying outside the class they introduced. As it turns out, also the ideals of such examples satisfy the assumptions of Theorem \ref{thm:main}. It is therefore worth to ask whether Theorem \ref{thm:main} can be reversed for square-free monomial ideals. Also, we could not find any square-free monomial ideal satisfying (iii) but not (i). All this is discussed after Lemma \ref{lem:retract}. 

\vskip 2mm

{\it Acknowledgments}: We wish to thank the anonymous referee for carefully reading the paper and giving fundamental advice.

\section{Basics on blow-up algebras}

Let $S=K[x_1,\ldots ,x_n]$ be a polynomial ring in $n$ variables over a field $K$, $\mm$ be the maximal irrelevant ideal of $S$, and $I$ a homogeneous ideal of $S$. We will intensively work with the following {\it blow-up algebras}:
\begin{itemize}
\item The {\it Rees algebra} of $I$, $R(I):=\bigoplus_{k\geq 0}I^k$.
\item The {\it associated graded ring} of $I$, $G(I):=\bigoplus_{k\geq 0}I^k/I^{k+1}$.
\item The {\it fiber cone} of $I$, $F(I):=\bigoplus_{k\geq 0}I^k/\mm I^k\cong G(I)/\mm G(I)$.
\end{itemize}
(All the direct sums are taken as $S$-modules). Recall that $\dim(R(I))=n+1$, $\dim(G(I))=n$ and $\ell(I):=\dim(F(I))$ is called the {\it analytic spread} of $I$.

\begin{remark}
Notice that, if $I=(f_1,\ldots ,f_r)$ where the $f_i$'s are forms of the same degree, then
\[F(I)\cong K[f_1,\ldots ,f_r].\]
\end{remark}

Many properties of the powers of $I$ are reflected by the blow-up algebras. In this paper we are interested in studying the {\it depth-function} of $I$:
\[k\mapsto \depth(S/I^k),\]
so let us see how to relate the depth-function with the blow-up algebras: Notice that, for all $i\geq 0$, we have isomorphism of $S$-modules:
\[H_{\mm G(I)}^i(G(I))\cong H_{\mm}^i(G(I))\cong \bigoplus_{k\geq 0}H_{\mm}^i(I^k/I^{k+1}),\]
where the first isomorphism follows by the independence of the base in computing local cohomology, while the second one holds true because local cohomology commutes with direct sums. Consequently:
\[\grade(\mm G(I),G(I))=\min_{k\geq 0}\{\depth(I^k/I^{k+1})\}=\min_{k\geq 1}\{\depth(S/I^k)\},\]
where the last equality follows from the short exact sequences 
\[0\rightarrow I^k/I^{k+1}\rightarrow S/I^{k+1}\rightarrow S/I^k\rightarrow 0 \ \ \ \mbox{where } \ k\geq 0.\] 
Since $\grade(\mm G(I),G(I))\leq \height(\mm G(I))$, it follows an inequality due to Burch in \cite{Bu}:
\[\ell(I)\leq n-\min_{k\geq 1}\{\depth(S/I^k)\}.\]
If $G(I)$ is Cohen-Macaulay, then $\grade(\mm G(I),G(I))= \height(\mm G(I))$, so that:
\begin{equation}\label{eq:eh1}
\ell(I) = n-\min_{k\geq 1}\{\depth(S/I^k)\}.
\end{equation}
The equality above is due to Eisenbud and Huneke in \cite[Proposition 3.3]{EH}. The argument used there is different from the above one, yielding the following interesting further property (under the assumption $G(I)$ is Cohen-Macaulay):
\begin{equation}\label{eq:eh2}
\depth(S/I^s)=\min_{k\geq 1}\{\depth(S/I^k)\}\implies \depth(S/I^{s+1})=\depth(S/I^s).
\end{equation}
In view of the above discussion, it is relevant to our purposes to understand when $G(I)$ is Cohen-Macaulay. Therefore, let us recall the following beautiful result of Lipman \cite[Theorem 5]{Li}:

\begin{theorem}\label{thm:lipman}
$G(I)$ is Cohen-Macaulay if and only if $R(I)$ is Cohen-Macaulay.
\end{theorem}

\section{The general result}

\begin{lemma}\label{lem:inequality}
Let $I=(f_1,\ldots ,f_r)\subseteq S$ be a homogeneous ideal, $A$ the $K$-subalgebra of $S$ generated by $f_1,\ldots ,f_r$ and $\nn=I\cap A$. Then we have
\[\dim(A)=\height(\nn)\geq \dim(F(I)).\]
\end{lemma}
\begin{proof}
The first equality follows just because $\nn$ is a maximal ideal of $A$ and in a domain which is a finitely generated $K$-algebra all the maximal ideals have the same height. 

If $R=\bigoplus_{k\geq 0}\nn^k/\nn^{k+1}$ is the associated graded ring of $A$ with respect to $\nn$, then $\dim(R)=\height(\nn)$. We are going to show that the Hilbert function of $R$ evaluated at $k$ is at least as the Hilbert function of $F(I)$ evaluated at $k$ for each $k\in\NN$.

The $K$-vector space $\nn^k/\nn^{k+1}$ is generated by the (images of the) elements 
\[f_{i_1}\cdots f_{i_k}, \ \ \ 1\leq i_1,\ldots ,i_k\leq r.\](The images of) such elements also generate the $K$-vector space $I^k/\mm I^k$.
Furthermore, if 
\[g=\sum_{1\leq i_1,\ldots ,i_k\leq r}a_{i_1,\ldots ,i_r}\cdot f_{i_1}\cdots f_{i_k}\in\nn^{k+1}\] for some scalars $a_{i_1,\ldots ,i_r}\in K$, then $g\in\mm I^k$ too (because $(f_1,\ldots ,f_r)\subseteq \mm$). So, we conclude that
\[\dim_K(\nn^k/\nn^{k+1})\geq \dim_K(I^k/\mm I^k) \ \ \ \forall \ k\in\NN.\]In particular, $\dim(R)\geq \dim(F(I))$.
\end{proof}
%
%
%
\begin{definition}
We say that a homogeneous ideal $I\subseteq S$ is a {\it summand ideal} if there exists a system of homogeneous generators $f_1\ldots ,f_r$ of $I$ such that the $K$-algebra $A=K[f_1,\ldots,f_r]$ is a direct summand of $S$, that is there exists a $K$-vector subspace $B$ of $S$ such that $B$ is an $A$-module and $S=A\oplus B$.
\end{definition}

\begin{lemma}
If $I\subseteq S$ is a {\it summand ideal}, then there exists a minimal system of homogeneous generators $f_1,\ldots ,f_r$ of $I$ such that the $K$-algebra $K[f_1,\ldots ,f_r]$ is a direct summand of~$S$. 
\end{lemma}
\begin{proof}
Assume that $g_1,\ldots ,g_k$ generate $I$, $A'=K[g_1,\ldots,g_k]$ is a direct summand of $S$, and $B'$ is the $A'$ module such that $S=A'\oplus B'$. If $g_1,\ldots ,g_k$ is not a minimal system of generators of $I$, then we can assume that:
\[g_k=\sum_{i=1}^{k-1}h_{i}g_{i}=\sum_{i=1}^{k-1}a_{i}g_{i}+\sum_{i=1}^{k-1}b_{i}g_{i},\]
where the $h_{i}$ are polynomials of $S$ and $a_i\in A'$, $b_i\in B'$ are the elements such that $h_i=a_i+b_i$. If $\sum_{i=1}^{k-1}b_{i}g_{i}=0$, then $A'=K[g_1,\ldots ,g_{k-1}]$, so we can conclude by induction. Otherwise, $\sum_{i=1}^{k-1}b_{i}g_{i}$ is a nonzero element of $A'\cap B'$, that is a contradiction. 
\end{proof}

\begin{theorem}\label{thm:main}
Let $I$ be a homogeneous ideal of $S$ such that:
\begin{itemize}
\item[(i)] $G(I)$ (or equivalently $R(I)$) is Cohen-Macaulay;
\item[(ii)] $H_I^i(S)=0$ for any $i>\pdim(S/I)$;
\item[(iii)] $I$ is a summand ideal.
\end{itemize}
Then the depth-function of $I$ is constant.
\end{theorem}
\begin{proof}
Let $f_1,\ldots ,f_r$ be homogeneous generators of $I$ such that $A=K[f_1,\ldots ,f_r]$ is a direct summand of~$S$. Let us call $\nn = I\cap A$ and $d:=\height(\nn)=\dim(A)$. By Grothendieck's nonvanishing theorem we have
\[H_{\nn}^d(A)\neq 0.\]
Since there exists an $A$-module $B$ such that $S=A\oplus B$ and the local cohomology is an additive functor, we infer
\[H_{\nn}^d(S)\cong H_{\nn}^d(A)\oplus H_{\nn}^d(B)\neq 0.\]
Therefore, because $\nn S=I$:
\[H_{I}^d(S)\neq 0.\]
By assumption, we therefore infer that $d\leq \pdim(S/I)$. However, by Lemma \ref{lem:inequality} we know that $d\geq \dim(F(I))$, and because $G(I)$ is Cohen-Macaulay $\dim(F(I))=\max_{m\geq 1}\{\pdim(S/I^m)\}$ by \eqref{eq:eh1}. Therefore 
\[\pdim(S/I)=\max_{m\geq 1}\{\pdim(S/I^m)\},\] 
that using \eqref{eq:eh2} implies $\depth(S/I)=\depth(S/I^m) \ \ \forall \ m\geq 1$. 
\end{proof}

The above theorem has strong assumptions, however let us recall that the second condition is satisfied by a broad class of homogeneous ideals $I$ of $S$.

\begin{proposition}\label{prop:notsorare}
We have $H_I^i(S)=0$ for any $i>\pdim(S/I)$ in the following cases:
\begin{itemize}
\item[(a)] $\depth(S/I)\leq 3$ (Varbaro \cite{Va});
\item[(b)] The characteristic of the field $K$ is positive (Peskine-Szpiro \cite{PS});
\item[(c)] $I$ is a monomial ideal (Lyubeznik \cite{Ly}).
\end{itemize}
\end{proposition}

\begin{example}
The hypotheses in the above proposition are necessary, indeed the ideal $I\subseteq S=\mathbb{C}[x_1,\ldots ,x_6]$ generated by the 2-minors of a $2\times 3$ generic matrix is a binomial ideal such that $\depth(S/I)=4$ and $H_I^3(S)\neq 0$ (see \cite{BS}).
\end{example}

\begin{example}
The following is a quite interesting example: Take an $r\times s$ matrix (say $r\leq s$) whose entries are indeterminates over $K$, and consider the ideal $I\subseteq S=K[X]$ generated by the $r$-minors of $X$. Let us see if the assumptions of Theorem \ref{thm:main} are satisfied by $I$:
\begin{itemize}
\item[(i)] $G(I)$ (as well as $R(I)$) is Cohen-Macaulay.
\item[(ii)] If $\chara(K)>0$, then $H_I^i(S)=0$ for any $i>\pdim(S/I)$ by \cite{PS}.
\item[(iii)] If $\chara(K)=0$, then $I$ is a summand ideal. Indeed, the $K$-algebra $A$ generated by the $r$-minors of $X$ is an $\mathrm{SL}(r,K)$-invariant subring of $S$, thus ($\mathrm{SL}(r,K)$ being linearly reductive in characteristic 0) it admits a Reynolds operator (cf. \cite[Theorem 2.2.5]{DK}).
\end{itemize}
However the depth-function of $I$ is not constant (independently on the characteristic). More precisely, the arguments used by Akin, Buchsbaum and Weyman in \cite{ABW} yield:
\[\depth(S/I^k)=rs-\min\{k,r\}(s-r)-1\]
(see \cite[Remark 3.2]{BCV} for the explicit proof). So, accordingly to the characteristic, the remaining assumption of Theorem \ref{thm:main} has to fail. That is:
\begin{itemize}
\item If $\chara(K)=0$, then $H_I^i(S)\neq 0$ for some $i>\pdim(S/I)=s-r+1$. (Indeed this is well known, since $H_I^{r(s-r)+1}(S)\neq 0$ by \cite{BS}). 
\item If $\chara(K)>0$, then $I$ is not a summand ideal.
\end{itemize}
\end{example}

One could wonder if the assumption (iii) of Theorem \ref{thm:main} implies condition (i). As shown in the following example, this is not the case, even for monomial~ideals.

\begin{example}\label{ex:no}
Consider the monomial ideal 
\[I=(u_1=x_1x_4^3,u_2=x_2x_5^3,u_3=x_3x_4x_5x_6)\subseteq S=K[x_1,\ldots ,x_6].\]
By Lemma \ref{lem:retract}, the algebra $A=K[u_1,u_2,u_3]$ (which in this case coincides with the fiber cone $F(I)$), is an algebra retract of $S$; in particular, $A$ is a direct summand of $S$. One can check by using \cite{cocoa} that the $h$-vector of the Rees algebra of $I$ is:
\[(1,2,3,4,3,1,-1).\]
In particular, $R(I)$ is not Cohen-Macaulay, so $G(I)$ is not Cohen-Macaulay as well by Lipman's Theorem \ref{thm:lipman}. Again by using \cite{cocoa}, one can check that $\dim(S/I)=4$ and $\depth(S/I^k)=3 \ \forall \ k\leq 20$. 
\end{example} 

Always on this kind of consideration, we have the following, quite not intuitive, corollary:

\begin{corollary}
Assume that $\chara(K)>0$. Let $I\subseteq S$ be a homogeneous radical ideal such that $S/I$ is Cohen-Macaulay but $I$ is not a complete intersection. Then:
\[I \mbox{ is a summand ideal } \implies G(I) \mbox{ is not Cohen-Macaulay.}\]
\end{corollary}
\begin{proof}
If $G(I)$ were Cohen-Macaulay, then by using together Theorem \ref{thm:main} and Proposition \ref{prop:notsorare}, we would have that
\[\depth(S/I^k)=\depth(S/I)=\dim(S/I) \ \ \ \forall \ k\geq 1.\]
Because $I$ is radical, this would be possible only if $I$ was a complete intersection by a result in \cite{CN}.
\end{proof}

We conclude this section by introducing a concrete class of summand ideals.

\begin{definition}
Suppose that $P_1,\ldots ,P_s$ is a partition of $\{x_1,\ldots ,x_n\}$ and $\deg(x_j)={\bf e_i}\in\ZZ^s$ if and only if $x_j\in P_i$. This supplies a $\ZZ^s$-graded structure to $S$. Given a subgroup $H\subseteq \ZZ^s$, let ${\bf a_1},\ldots ,{\bf a_k}$ be a minimal system of generators of the monoid $H\cap \NN^s$. The ideal $I_H\subseteq S$ generated by all polynomials of multi-degree ${\bf a_1}, \ldots , {\bf a_k}$, is called a {\it degree-selection} ideal. 
\end{definition}

\begin{proposition}\label{propnew1}
Any degree-selection ideal is a summand monomial ideal.
\end{proposition}
\begin{proof}
Obviously a degree-selection ideal $I=I_H$ must be a monomial ideal. If $I$ is minimally generated by multi-homogeneous polynomials $f_1,\ldots ,f_r$, then:
\[A=K[f_1,\ldots ,f_r]=\bigoplus_{{\bf v}\in H}S_{{\bf v}}.\]
So $B=\oplus_{{\bf v}\not\in H}S_{{\bf v}}$ is an $A$-module, and $S=A\oplus B$.
\end{proof}

Notice that if $s=1$, any degree-selection ideal is a power of the irrelevant maximal ideal, and the corresponding algebra $A$ is some Veronese subalgebra of~$S$. 

\section{The monomial case}
In the monomial case, thanks to Proposition \ref{prop:notsorare}, Theorem \ref{thm:main} can be stated as follows:

\begin{theorem}\label{thm:main-mon-version}
Let $I$ be a monomial ideal of $S$ satisfying the two conditions below:
\begin{itemize}
\item[(i)] $G(I)$ (or equivalently $R(I)$) is Cohen-Macaulay;
\item[(ii)] $I$ is a summand ideal.
\end{itemize}
Then the depth-function of $I$ is constant.
\end{theorem}

As shown by Example \ref{ex:no}, even under the assumptions of the above theorem (ii) does not imply (i). It would be desirable, though, to have a characterization of $A=K[u_1,\ldots ,u_r]$ being a direct summand of $S$, where $u_1,\ldots ,u_r$ is the minimal system of monomial generators of $I$. 

\begin{remark}
In the above situation, if $A$ is a direct summand of $S$, then the $K$-vector space $B$ generated by all the monomials of $S\setminus A$ is an $A$-module, and $A\oplus B=S$. To see this, let $C$ be an $A$-module such that $A\oplus C=S$. Then any nonzero monomial $b\in B$ can be written uniquely as $b=a+c$, where $a\in A$ and $c$ is a nonzero element of $C$. Therefore $a'b=a'a+a'c$ is a monomial of $S\setminus A$ for any nonzero monomial $a'\in A$; this implies that $B$ is an $A$-module.
\end{remark}

By keeping the above remark in mind, if we associate to each $u_i=x_1^{a_{1i}}\cdots x_n^{a_{ni}}$ the vector ${\bf a_i}=(a_{1i},\ldots ,a_{ni})\in\ZZ^n$, and denote by $\mathcal{C}=\N\{{\bf a_1},\ldots , {\bf a_r}\}\subseteq \ZZ^n$, it is immediate to verify that:
\begin{equation}\label{directsummands}
A \mbox{ is a direct summand of $S$ }\iff \ZZ\mathcal{C}\cap \N^n=\mathcal{C}.
\end{equation}
The above characterization is not very satisfactory, since it is not instantaneous to detect from the monomial generators $u_1,\ldots ,u_r$. 
In contrast, the shape of $u_1,\ldots ,u_r$ can be explicitly described to characterize when $A$ is a direct summand of $S$ as a ring (which is a stronger property than being a direct summand as an $A$-module).

\begin{definition}
A ring inclusion $\iota:R'\hookrightarrow R$ is an {\it algebra retract} if there is a ring homomorphism $\pi:R\rightarrow R'$ such that $\pi\circ \iota =1_{R'}$.
\end{definition}

We found a proof of Lemma \ref{lem:retract} in a discussion on MathOverflow. The proof is due to Zaimi \cite{Za}, we report it here for the convenience of the reader. Before a remark:

\begin{remark}
In the MathOverflow debate mentioned above it is also discussed the case in which $A$ is (isomorphic to) a direct summand of some polynomial ring as an $A$-module. This is the case if and only if $A$ is normal (\cite[Proposition 1]{Ho}), which, with the notation of \eqref{directsummands}, is the case if and only if
\[\ZZ\mathcal{C}\cap\QQ_{\geq 0}\mathcal{C}=\mathcal{C}.\]
Be careful! We are interested in the case in which $A$ is a direct summand of $S$, and not just isomorphic to a direct summand of a polynomial ring. For example, $A=K[xy,xz,yz]$ is not a direct summand of $S=K[x,y,z]$, however $A$ is isomorphic (as a $K$-algebra) to $S$.
\end{remark}

\begin{lemma}\label{lem:retract}
Given a monomial ideal $I\subseteq S$ minimally generated by the monomials $u_1,\ldots ,u_r$, the inclusion $K[u_1,\ldots ,u_r]\subseteq S$ is an algebra retract if and only if there is an $r$-subset $U=\{\ell_1,\ldots ,\ell_r\}\subseteq \{1,\ldots ,n\}$ such that, for each $i=1,\ldots ,r$:
\[u_i=x_{\ell_i}v_i \ \ \ \mbox{where } \  v_i\in K[x_j:j\in \{1,\ldots ,n\}\setminus U]. \]
\end{lemma}
\begin{proof}
Let $A=K[u_1,\ldots ,u_r]$ and $\iota:A\hookrightarrow S$ be the natural inclusion. If there is a subset $U\subseteq \{1,\ldots ,n\}$ as in the statement, then the homomorphism of $K$-algebras $\pi:S\rightarrow A$ obtained by extending the rule
\begin{align*}
\pi(x_j)=\begin{cases}
u_i, & \mbox{if } j=\ell_i; \\
1, & \mbox{if } j\notin U
\end{cases}
\end{align*}
satisfies $\pi\circ\iota =1_A$.

On the other hand, if $\iota:A\hookrightarrow S$ is an algebra retract, then there is a ring homomorphism $\pi:S\rightarrow A$ such that $\pi\circ\iota =1_A$. So, for any $i=1,\ldots ,r$:
\[u_i=\pi(u_i)=\prod_{j=1}^n\pi(x_j)^{a_{i,j}}, \ \ \ \mbox{ where }u_i=\prod_{j=1}^nx_j^{a_{i,j}}.\]
But, forming the $u_i$'s a minimal set of generators of $I$, they also are minimal generators of $A$ as a $K$-algebra. So, for any $i=1,\ldots ,r$, there exists $\ell_i\in\{1,\ldots ,n\}$ such that $\pi(x_{\ell_i})=\lambda u_i$ for $\lambda\in K$, $a_{i,\ell_i}=1$ and $\pi(x_j)\in K$ whenever $j\neq \ell_i$ and $a_{i,j}>0$. This lets us conclude.
\end{proof}

In \cite{HV}, Herzog and Vladoiu investigated on the square-free monomial ideals with constant depth-function. Among other things, they presented various classes of such ideals, as well as examples of square-free monomial ideals with constant depth-function not falling within their classes. Indeed, a characterization of such ideals is still missing. Below we will notice that all the classes and examples of square-free monomial ideals with constant depth-function provided in \cite{HV} satisfy the hypotheses of Theorem \ref{thm:main-mon-version}.

\begin{proposition}\label{propnew2}
With the notation of \cite[Corollary 1.2]{HV}, $I$ is a degree-selection ideal corresponding to the submonoid $M\subseteq \NN^s$ generated by:
\[\sum_{i\in A_1}{\bf e_i}, \ \ \sum_{i\in A_2}{\bf e_i}, \ \ldots \ , \sum_{i\in A_r}{\bf e_i}, \]
where the $\ZZ^s$-graded structure on $S$ is given by $\deg(x_j)={\bf e_i}$ if and only if $x_j\in P_i$.
\end{proposition}
\begin{proof}
We will make a double induction on $r$ and the maximum cardinality of the subsets $A_i\subseteq \{1,\ldots ,s\}$: if each $A_i$ has cardinality 1 the statement is trivial, as well as if $r=1$.

By the assumptions of \cite[Corollary 1.2]{HV}, there is an index $j\in\{1,\ldots ,s\}$ such that:
\[\left(\bigcup_{j\in A_i}A_i\setminus \{j\}\right)\bigcap \left(\bigcup_{j\not\in A_i}A_i\right)=\emptyset ,\]
and the two collections $\{A_i\setminus \{j\}:j\in A_i\}$ and $\{A_i:j\notin A_i\}$ are in the family $\mathcal{A}$. Let $u_1,\ldots ,u_q$ be a minimal system of monomial generators of $I$, and $v\in S$ some nonzero monomial. 

\vspace{1mm}

{\bf Claim}: The inclusion $v\in K[u_1,\ldots ,u_q]$ holds if and only if $\deg(v)\in M$.

\vspace{1mm}

\noindent To see this, let $v'$ be the larger degree unique monic monomial of $K[P_j]$ dividing $v$, and let 
\[A(j)=\{i\in \{1,\ldots ,r\}:j\in A_i\}.\]
If $v'=1$, then $\deg(v)\in \langle \{\sum_{i\in A_k}{\bf e_i}:k\not\in A(j)\} \rangle$ if and only if $v\in K[u_k:k\not\in A(j)]$ by induction. Otherwise, $\deg(v/v')\in \langle \{\sum_{i\in A_k\setminus\{j\}}{\bf e_i}:k\in A(j) \}\rangle$ if and only if $v/v'\in K[u_k/u_k':k\in A(j)]$, where $u_i'$ is the larger degree unique monic monomial of $K[P_j]$ dividing $u_i$, again by induction. These facts let us prove the claim, and thus conclude.
\end{proof}

\begin{itemize}
\item[(i)] The class of square-free monomial ideals introduced in \cite[Example 1.3 (ii)]{HV}, and more generally the class of \cite[Corollary 1.2]{HV}, are summand ideals by Propositions \ref{propnew2} and \ref{propnew1}: furthermore, these ideals are obtained recursively by summing up or multiplying ideals generated in disjoint sets of variables, with ideals generated by variables as starting point.  Such ideals have a Cohen-Macaulay Rees algebra by \cite[Theorem 4.7]{SVV1} and \cite[Corollary 2.10]{Hy}. So, such ideals satisfy the hypotheses of Theorem \ref{thm:main-mon-version}.


\item[(ii)] The square-free monomial ideals of \cite[Theorems 2.2, 2.5 and 2.6]{HV}, being in the class introduced in \cite[Example 1.3 (ii)]{HV}, satisfy the hypotheses of Theorem \ref{thm:main-mon-version} by the previous point.

\item[(iii)] The ideal $I=(x_1x_2x_3,x_3x_4x_5,x_1x_5x_6)\subseteq S=K[x_1,\ldots ,x_6]$ of \cite[Example 1.4]{HV} is such that $R(I)$ is Cohen-Macaulay (this can be checked by using \cite{cocoa}). Furthermore $K[x_1x_2x_3,x_3x_4x_5,x_1x_5x_6]$, by Lemma \ref{lem:retract}, is a direct summand (indeed an algebra retract) of $S$.

\item[(iv)] Similar thing as in (iii) happen to \cite[Example 2.7 (i)-(ii)]{HV}.
%

\end{itemize}

The following questions arise naturally, and hopefully will give further motivations to study this topic. 

\begin{question}\label{q1}
Let $I\subseteq S$ be a square-free monomial ideal minimally generated by the monomials $u_1,\ldots ,u_r$. If $A=K[u_1,\ldots ,u_r]$ is a direct summand of $S$ as an $A$-module, is the Rees algebra $R(I)$ Cohen-Macaulay?
\end{question}

\begin{question}\label{q2}
Does the converse of Theorem \ref{thm:main-mon-version} holds for square-free monomial ideals?
\end{question}

Note that questions \ref{q1} and \ref{q2} have a negative answer for non-square-free monomial ideals by Example \ref{ex:no}. However, the depth-functions of square-free monomial ideals seem to have a much more rigid behavior than the depth-functions of arbitrary monomial ideals, and also if the above questions had a negative answer in general, maybe they have a positive answer in some special cases, e.g. if the monomial ideal is generated in a single degree. To this purpose, it is worth to notice that all the above questions have a positive answer for square-free monomial ideals generated in degree~2:

\begin{proposition}
For square-free monomial ideals generated in degree 2, questions \ref{q1} and \ref{q2} have a positive answer. 
\end{proposition}
\begin{proof}
Any square-free monomial ideal $I\subseteq S$ generated in degree 2 is associated to a simple graph $G$ in $n$ vertices by the rule $I=I(G)=(x_ix_j:\{i,j\} \ \mbox{ is an edge of }G)$. Also, denote by $K[G]=K[x_ix_j:\{i,j\} \ \mbox{ is an edge of }G]$. If $K[G]$ is a direct summand of $S$ as a $K[G]$-module, then $K[H]$ is a direct summand of $K[x_i:i\in H]$ as a $K[H]$-module for each connected component $H$ of $G$. In particular, $K[H]$ is normal by \cite[Proposition 1]{Ho}. Thus the Rees algebra $R(I(H))$ is also normal by \cite[Corollary 2.8]{SVV}. Therefore $R(I(H))$ is Cohen-Macaulay for each connected component $H$ of $G$ by \cite[Theorem 1]{Ho}, and thus $R(I(G))$ is Cohen-Macaulay by \cite[Theorem 1.1]{HV}. This gives an affirmative answer to Question \ref{q1}.

An affirmative answer to Question \ref{q2} is straightforward by the characterization of quadratic square-free monomial ideals with constant depth-function given in \cite[Theorem 2.2]{HV}.

\end{proof}

\end{document}